\documentclass[12pt,leqno]{amsart}
\usepackage{amsfonts,amsthm,amsmath,mathtools,amssymb,comment,xcolor}

\theoremstyle{plain}
\newtheorem{thm}{Theorem}[section]

\newtheorem{lem}{Lemma}[section]
\newtheorem{cor}{Corollary}[section]

\theoremstyle{definition}
\newtheorem{df}{Definition}[section]
\newtheorem{rem}{Remark}[section]
\newtheorem{ex}{Example}[section]

\newcommand{\FF}{\mathbb{F}}
\newcommand{\ZZ}{\mathbb{Z}}
\newcommand{\RR}{\mathbb{R}}
\newcommand{\CC}{\mathbb{C}}

\newcommand{\QQ}{\mathbb{Q}}

\newcommand{\HH}{\mathbb{H}}

\newcommand{\CJ}{\mathrm{CJ}}

\newcommand{\cwe}{\mathbf{cwe}}

\newcommand{\la}{\langle}
\newcommand{\ra}{\rangle}

\def\bm#1{\mathbf{#1}}

\DeclareMathOperator{\supp}{supp}

\DeclareMathOperator{\wt}{wt}

\DeclareMathOperator{\Harm}{Harm}
\DeclareMathOperator{\Gram}{Gram}

\DeclareMathOperator{\Hom}{Hom}
\DeclareMathOperator{\Mat}{Mat}
\DeclareMathOperator{\Ima}{Im}
\DeclareMathOperator{\tr}{tr}
\DeclareMathOperator{\diag}{diag}

\begin{document}

\title
[Equivariant theory]
{Equivariant theory for codes and lattices I}

\author[Chakraborty]{Himadri Shekhar Chakraborty*}
\thanks {*Corresponding author}
\address
	{
		Department of Mathematics\\ 
		Shahjalal University of Science and Technology\\ 
		Sylhet-3114, Bangladesh 
	}
\email{himadri-mat@sust.edu}

\author[Miezaki]{Tsuyoshi Miezaki}
\address
	{
		Faculty of Science and Engineering\\
		Waseda University\\ 
		Tokyo 903-0213, Japan
	}
\email{miezaki@waseda.jp} 


\date{\today}
\maketitle

\begin{abstract}
	In this paper,
	we present a generalization of Hayden's theorem~\cite[Theorem 4.2]{BrHalHay1981} 
	for $G$-codes over finite Frobenius rings.
	A lattice theoretical form of this generalization is also given.
	Moreover, 
	Astumi's MacWilliams identity~\cite[Theorem 1]{Atsumi1995} is generalized
	in several ways for different weight enumerators
	of~$G$-codes over finite Frobenius rings.
	Furthermore, 
	we provide the Jacobi analogue of Astumi's MacWilliams identity for~$G$-codes
	over finite Frobenius rings.
	Finally, we study the relation between $G$-codes and its corresponding $G$-lattices.
\end{abstract}

{\small
	\vspace{-0.15in}
	
	\quad
}


\section{Introduction}

Recently, the authors in several articles~\cite{Atsumi1995, BrHalHay1981, Yoshida1988} expressed their interest in the study of linear codes over~$\FF_{q}$ with group action.
In~\cite[Theorem 4.2]{BrHalHay1981}, Hayden showed for a $G$-code~$C$ over~$\FF_{q}$ that 
$(C\theta)^{\perp} = \ker \theta \oplus C^{\perp}\theta$,
where $G$ is an automorphism group such that 
$|G|$ has an inverse in~$\FF_{q}$ 
and~$\theta$ is an operator defined as:
\[
	\theta 
	:= 
	\frac{1}{|G|}
	\sum_{g\in G}
	g.
\]
Now this theorem is called Hayden's theorem and the operator~$\theta$ is called
Hayden's operator.
Yoshida~\cite{Yoshida1988} gave a generalization of the MacWilliams 
identity for $G$-codes over~$\FF_{q}$, where $G$ is a permutation group
such that $|G|$ is prime to~$q$.
Later, Atsumi~\cite{Atsumi1995} presented an analogue of Yoshida's MacWilliams
identity for $G$-codes with an action of Hayden's operator.
Consecutively, Atsumi~\cite{Atsumi1998} introduced the notion of $G$-lattices
and gave a lattice version of Hayden's theorem.

In this paper, we introduce a Hayden's like operator, and 
present a generalization of Hayden's theorem, firstly for $G$-codes 
over finite Frobenius rings (see~\cite{Wood1999}) and then for
$G$-lattices.
We also generalized Atsumi's MacWilliams identity for
weight enumerators of $G$-codes over finite Frobenius rings.
Bachoc~\cite{Bachoc} gave a harmonic generalization of the MacWilliams
identity for codes associated to discrete harmonic functions.
We refer the readers to~\cite{BachocNonBinary, BrChIsMiTa2023, CMO2023, Tanabe2001}
for more discussions on codes associated to discrete harmonic functions.
We present the harmonic generalization of Atsumi's MacWilliams identity
for $G$-codes over finite Frobenius rings. 
The class of higher genus weight enumerators in coding theory were widely 
studied in many papers, firstly over~$\FF_{q}$ and later over~$\ZZ_{k}$;
for instance, see~\cite{BDHO1999, CM2021, CMO2022, CMO2024, MO2019}.
In this study, we give a higher genus version of generalization of
Atsumi's MacWilliams identity for $G$-codes over finite Frobenius rings. 
Ozeki~\cite{Ozeki} introduced the notion of Jacobi polynomials for codes
in analogy with Jacobi theta series of lattices. 
Later, several forms of generalizations 
of the Jacobi polynomials for codes 
were studied in~\cite{CM2022, CMO2022, CMOT2023, CIT20xx}. 
We prove the Jacobi polynomials analogue of Atsumi's MacWilliams identity
for $G$-codes over finite Frobenius rings.
Many authors studied the properties of
lattices that constructed from codes. 
In particular, articles like as~\cite{BDHO1999, Munemasa, Runge}
described the relation between the weight enumerators of codes
and the theta series of lattices associated with the codes. 
Moreover, Bannai and Ozeki~\cite{BO1996} studied the Jacobi
analogue of this relation. 
In this paper, we study these two types of relation
for $G$-codes and associated $G$-lattices.

This paper is organised as follows.
In Section~\ref{Sec:Preli},
we recall preliminary definitions
and results on 
finite Frobenius rings, codes, and lattices.
In Section~\ref{Sec:Gcodes},
we give an analogue of Hayden's theorem (Theorem~\ref{Thm:ThetaH}).
We also give a useful result the duals of $G$-codes over finite Frobenius rings 
(Theorems~\ref{Thm:CodeThetaMatrix}). 
We also present the generalizations of Atsumi's MacWilliams identity 
for various types of weight polynomials for $G$-codes over finite Frobenius rings
(Theorems~\ref{Thm:EquivMacWilliams}, \ref{Thm:CompWeightMacWilliams}, \ref{Thm:gthCWEMacWilliams}, \ref{Thm:HarmMacWilliams}). 
In Section~\ref{Sec:Glattices},
we prove the $G$-lattice version of Hayden's theorem and Jacobi formula 
(Theorems~\ref{Thm:LatticeHayden}, \ref{Thm:GLatticeJacobiFormula}).
In Section~\ref{Sec:GcodetoGlattice},
we study the properties of $G$-lattices associated to $G$-codes
(Theorems~\ref{Thm:GcodeGlattice1}, \ref{Thm:GcodeGlattice2}, \ref{Thm:WeightEnumtoThetaSeries}).
In Section~\ref{Sec:JacobiPolyTheta}, 
we present the Jacobi analogues of some results in
previous sections~(Theorems~\ref{Thm:EquivJacobiMacWilliams}, \ref{Thm:JacPolytoThetaSeries}).

\section{Preliminaries}\label{Sec:Preli}

In this section, 
we briefly discussion the definitions and basic
properties of 
finite Frobenius rings, codes and lattices 
that are frequently used in this paper. 
For the detail discussions on these topics, 
we refer the readers 
to~\cite{CS1999, Dougherty, HP2003, Lam1999, Wood1999}.

\subsection{Finite Frobenius rings and its characters}\label{SubSec:FrobeniusRings}

Let $R$ be a finite ring with identity satisfying the associative property. 
Then $R$ is called a \emph{Frobenius ring} 
if it satisfies the following conditions:
\begin{itemize}
	\item [(i)]
	$R/\textrm{Rad}(R) \cong \textrm{Soc}(R_R)$ as a right $R$-module;
	
	\item [(ii)]
	$R/\textrm{Rad}(R) \cong \textrm{Soc}(_R R)$ as a left $R$-module.
\end{itemize}
Here $_R M$ (or $M_R$) to denotes a left (or right) $R$-module $M$. 
Recall the Jacobson radical of $R$, denoted by $\textrm{Rad}(R)$,
and the socal of an $R$-module $M$, denoted by $\textrm{Soc(M)}$; 
for detailed information on $\textrm{Rad}(R)$ and $\textrm{Soc(M)}$, see~\cite{Lam1999}.
The class of finite Frobenius rings include the ring $\ZZ_{k}$ 
of integers modulo~$k$, 
finite fields and more generally finite chain rings.

Let $M$ be an $R$-module over a finite ring~$R$. 
In this paper, we consider characters as homomorphism 
$\chi : M \to \CC^{\ast}$ from additive group~$M$ to the 
multiplicative group of non-zero complex numbers~$\CC^{\ast}$
rather than map into $\QQ/\ZZ$.
For an $R$-module~$M$, 
we denote~$\widehat{M} := \Hom_{\ZZ}(M,\CC^{\ast})$ 
the \emph{character module} of~$M$.
If $M$ is a left (resp., right) $R$-module, 
then $\widehat{M}$ is a right (resp., left) $R$-module
via the module action
\[
	^{r}\chi(x) = \chi(xr)
	\quad
	(\mbox{resp. } \chi^{r}(x) = \chi(rx)),
\] 
where $\chi \in \widehat{M}$, $r \in R$ and $x \in M$ (see~\cite{Wood1999}).
Now we the following characterization of finite Frobenius rings in terms of 
character modules.

\begin{thm}\cite[Theorem 3.10]{Wood1999}\label{Thm:Character}
	Let $R$ be a finite ring.
	Then the following statements are equivalent:
	\begin{itemize}
		\item [(i)]
		$R$ is a Frobenius ring.
		\item [(ii)]
		As a left module, $\widehat{R} \cong \prescript{}{R}{R}$.
		\item [(iii)]
		As a right module, $\widehat{R} \cong {R}_{R}$.
	\end{itemize}
\end{thm}

Let $R$ be a finite ring. A character~$\chi$ of $R$ is called a
right (resp. left) \emph{generating character} 
if the mapping 
$\phi : R \to \widehat{R}$, defined by 
$\phi(r) = \chi^{r}$
(resp. $\phi(r) = \prescript{r}{}{\chi}$)
is an isomorphism of right (resp. left) $R$-modules. 
It is immediate from Theorem~\ref{Thm:Character} that
a finite ring~$R$ is Frobenius if and only if 
it admits a right or a left generating character.

\begin{ex}
	Consider the ring $\ZZ_{k}$ of integers modulo~$k$ 
	for some positive integer $k \ge 2$.
	Let~$\zeta_{k} := e^{2{\pi}i/k}$ be the
	primitive $k$-th root of unity.
	Then $\chi(a) = \zeta_{k}^{a}$
	for $a \in \ZZ_{k}$
	is a generating character. 
\end{ex}

Let $R$ be a finite Frobenius ring. Let $\chi$ be a generating character
associate with~$R$. Then for any $a \in R$,  we have the following property:
\[
	\sum_{b \in R}
	\chi(ab) 
	= 
	\begin{cases}
		|R| & \mbox{if} \quad a = 0, \\
		0 & \mbox{if} \quad a \neq 0.
	\end{cases} 
\]
We refer the readers to~\cite{Dougherty, HL2001, Wood1999} for more discussions 
on Frobenius rings and its characters.

\subsection{Linear codes over finite Frobenius rings}\label{SubSec:LinearCodes}
Let $R$ be the finite Frobenius ring.
Then $V:=R^{n}$ 
denotes the free $R$-module with ordinary inner product
\[
	(u,v) 
	:= 
	u_{1}v_{1} + \cdots + u_{n}v_{n},
\]
where
$u = (u_{1},\ldots,u_{n}) \in V$ and $v = (v_{1},\ldots,v_{n}) \in V$.
The \emph{support} and \emph{weight} of an element $u \in V$
can be defined as follow:
\begin{align*}
	\supp(u) 
	& := \{i \in [n] \mid u_{i} \neq 0\},\\
	\wt(u) 
	& := |\supp(u)|.
\end{align*}

Let $R$ be the finite Frobenius ring. 
Then a \emph{left} (or \emph{right}) $R$-\emph{linear code} 
of length~$n$ is a left (or right) $R$-submodule of $V$.
The elements of a code are known as codewords. 
We denote the \emph{left dual code} of a right $R$-linear code~$C$ by $^{\perp}C$ 
and the \emph{right dual code} of a left $R$-linear code~$D$ by $D^{\perp}$, 
and define as follows:
\begin{align*}
	^{\perp}C
	:=
	\{
	u \in V
	\mid
	(u, v) = 0
	\mbox{ for all }
	v \in C
	\},\\
	D^{\perp}
	:=
	\{
	u \in V
	\mid
	(v, u) = 0
	\mbox{ for all }
	v \in D
	\}.
\end{align*}
It is immediate from the above definition that
$^\perp C$ is a left $R$-submodule of $V$,
and $D^\perp$ is a right $R$-submodule of $V$.
A code $C$ (resp. $D$) is called \emph{left self-dual} 
(resp. \emph{right self-dual}) if $C = {^\perp}C$ (resp. $D = D^{\perp}$).

\begin{df}\label{DefWeightEnumerator}
	Let $R$ be the finite Frobenius ring.
	Let $C$ be a left (or right) $R$-linear code of length~$n$. 
	Then the \emph{weight enumerator} of $C$ is
	defined as follows:
	\[
	W_{C}(x,y) 
	:=
	\sum_{{\bf u}\in C}
	x^{n-\wt({u})}
	y^{\wt({u})}
	=
	\sum^{n}_{i=0} 
	A_{i} x^{n-i} y^{i},
	\]
	where $A_{i} := \#\{u \in C \mid \wt(u) = i\}$.
\end{df}

The following MacWilliams-type identity holds 
for the left linear codes over the finite Frobenius rings. 
Similarly,
most of the results in the subsequent sections are stated for left linear codes. 
Each of these statements can be re-expressed with equal validity for right linear codes.

\begin{thm}[MacWilliams-type identity]\label{Thm:MacWilliams} 
	Let $R$ be a finite Frobenius ring, 
	and $C$ be a left $R$-linear code of length~$n$. 
	Then 
	\[
	W_{^{\perp}C}
	(x,y)
	= 
	\frac{1}{|C|} 
	W_{C} 
	\left( 
	x+(|R|-1)y, 
	x-y 
	\right).
	\]
\end{thm}

\subsection{Lattices}

Let $\RR^{n}$ be the $n$-dimensional Euclidean space.
The \emph{inner product} of any two elements 
$u = (u_{1},\ldots,u_{n})$ and $v = (v_{1},\ldots,v_{n})$
can be defined as
\[
	\la u,v\ra
	:=
	u_{1}v_{1} + u_{2}v_{2}+\cdots+u_{n}v_{n}.
\]
An $n$-dimensional \emph{lattice} in $\RR^{n}$ is a subset 
$\Lambda \subseteq \RR^{n}$
with the property that there exists an $\RR$-basis
$({e}_{1},{e}_{2},\ldots,{e}_{n})$ of $\RR^{n}$
which is a $\ZZ$-basis of~$\Lambda$, that is,
\begin{equation*}\label{Equ:Lattice}
	\Lambda
	:=
	\ZZ {e}_{1}
	\oplus
	\ZZ {e}_{2}
	\oplus
	\cdots
	\oplus
	\ZZ {e}_{n}.
\end{equation*}
An $n \times n$ matrix~$M$ whose rows are the vectors 
$({e}_{1},{e}_{2}\ldots,{e}_{n})$ 
is called the \emph{generator matrix} of $\Lambda$.
The \emph{gram matrix} and the \emph{determinant} of a
lattice $\Lambda$ is defined as follows:
\begin{align*}
	\Gram(M) 
	& := 
	(\la {e}_{i}, {e}_{j}\ra) 
	=
	MM^{\tr},\\
	\det \Lambda 
	& := 
	|\det M|.
\end{align*}
A lattice~$\Lambda$ is called \emph{even} if
$\la u,u\ra \in 2\ZZ$
for all $u \in \Lambda$. 
The \emph{dual lattice} of $\Lambda$ is defined by
\[
	\Lambda^{\ast}
	:=
	\{
	u \in \RR^{n}
	\mid
	\la u,v\ra
	\in 
	\ZZ
	\mbox{ for all }
	v \in \Lambda 
	\}.
\]
A lattice~$\Lambda\subseteq \RR^{n}$ is called
\emph{integral} if $\Lambda \subseteq \Lambda^{\ast}$.
An integral lattice with $\Lambda = \Lambda^{\ast}$
is called \emph{unimodular}. 
	
\begin{df}
	Let $\Lambda$ be a lattice in $\RR^{n}$.
	Define 
	$\HH := \{z \in \CC \mid \rm{Im}(z) > 0\}$
	the upper half plane.
	Then the function $\Theta_{\Lambda} : \HH \to \CC$
	defined as follows 
	\[
		\Theta_{\Lambda}(z)
		:=
		\sum_{x \in \Lambda}
		q^{\la x, x \ra},
	\]
	where $q := e^{\pi i z}$ and $z \in \HH$,
	is called the \emph{theta series} of $\Lambda$.
\end{df}

The following Jacobi's formula hold for the theta series.

\begin{thm}[\cite{CS1999}]\label{Thm:JacobiFormula}
	Let $\Lambda$ be a lattice in $\RR^{n}$. 
	Then the following relation holds:
	\[
		\vartheta_{\Lambda^\ast}(z)
		=
		(\det \Lambda)^{1/2}
		(i/z)^{n/2}
		\vartheta_{\Lambda}(- 1/z).
	\]
\end{thm}

\section{Codes with group actions}\label{Sec:Gcodes}

Let $R$ be a finite Frobenius ring.
Let $G$ be a finite permutation group 
acting on the coordinates of~$V (=R^{n})$. 
Then $V$ is an $RG$-\emph{module} 
if there exists a mapping  
\[
	G \times V \to V;
	(g,v) \mapsto vg
\]
satisfying 
$vg \in V$ 
for all~$g \in G$ and $v \in V$.
Let~$v = (v_{1},\ldots,v_{n}) \in V$ and $g \in G$.
We define the action of $G$ on~$V$ as follows:
\[
	vg
	:=
	(x_{1},\ldots,x_{n})
	\mbox{ such that }
	x_{i} = v_{ig^{-1}}
	\mbox{ for }
	i = 1,\ldots,n.
\]
In this way, 
$V$ becomes an {free} $RG$-module. 
A $G$-\emph{code} is an left (or right) $RG$-submodule of $V$.
Now we have the following finite Frobenius ring analogue of 
Hayden's theorem~\cite[Theorem 4.2]{BrHalHay1981}. 
The proof of the following theorem is similar 
to Hayden's proof. So we omit it. 

\begin{thm}\label{Thm:ThetaH}
	Let $R$ be a finite Frobenius ring. 
	Assume that $G$ is a finite permutation group on 
	the coordinates of $R^{n}$.
	For any subgroup~$H$ of~$G$
	such that $|H|$ has an inverse in~$R$,
	we define an operator:
	\[
		\theta_{H}
		:=
		\frac{1}{|H|}
		\sum_{h\in H}
		h.
	\] 
	If $C$ is a left $G$-code of length~$n$ over $R$, 
	then we have
	\[
		^{\perp}(C\theta_{H})
		=
		\ker\theta_{H} 
		\oplus 
		{^\perp}C
		\theta_{H}.
	\]
\end{thm}

\begin{cor}
	Let $R$ be a finite Frobenius ring. 
	Assume that $G$ is a finite permutation group on 
	the coordinates of $R^{n}$
	provided that~$|G|$ has an inverse in~$R$.
	We define an operator:
	\[
		\theta
		:=
		\frac{1}{|G|}
		\sum_{g\in G}
		g.
	\] 
	If $C$ is a $G$-code of length~$n$ over $R$, 
	then we have
	\[
		^{\perp}(C\theta)
		=
		\ker\theta 
		\oplus 
		{^\perp}C
		\theta.
	\]
\end{cor}


Let $R$ be a finite Frobenius ring.
Let $\theta_{H}$ be the operator as described
in Theorem~\ref{Thm:ThetaH}.
Then for any $v \in V (= R^{n})$, we have
\[
	v\theta_{H}
	=
	\dfrac{1}{|H|}
	\sum_{h \in H}
	vh.
\]
Define 
$V\theta_{H} := \{v\theta_{H} \mid v \in V\}$.
Clearly, $V\theta_{H}$ is a left (or right) $R$-submodule of~$V$. 
Let
$H_{V}(\alpha_{1}),\ldots,H_{V}(\alpha_{t})$ 
be the orbits of the coordinates of $V$ under the action of $H$.
Let $m_{i}$ be the length of $H_{V}(\alpha_{i})$.
Therefore, assume that $H_{V}(\alpha_{i}) = \{a_{k} \in [n]\mid k = 1,\ldots,m_{i}\}$.
Now we form a diagonal matrix, 
$M_{H} := \diag(\ell_{1},\ell_{2},\ldots,\ell_{n})$ 
of order~$n$ corresponding to~$H$, where
$\ell_{j} = m_{i}$ if $j \in H_{V}(\alpha_{i})$.
We call this matrix as~\emph{$H$-orbit length matrix}. 
Define $\overline{H_{V}(\alpha_{i})} := (y_{1},\ldots,y_{n})$
such that
\[
	y_{j}
	:=
	\begin{cases}
		1 & \mbox{if } j = a_{k} \mbox{ for some } k;\\
		0 & \mbox{otherwise}.
	\end{cases}	
\]
Therefore, each element~$u \in V\theta_{H}$ 
can be written in the following form:
\begin{equation}\label{Equ:torbit}
	u	
	= 
	\sum_{i = 1}^{t}	
	u_{i} \overline{H_{V}(\alpha_{i})}.
\end{equation}
We call $\supp_{H}(u)$ the $H$-\emph{support} of 
$u \in V\theta_{H}$ is the set of $i$'s such 
that~$u_{i} \neq 0$ when
$u$ is written in form~(\ref{Equ:torbit}).
Then $\wt_{H}(u) = |\supp_{H}(u)|$ is the $H$-\emph{weight} 
of $u \in V\theta_{H}$.
If $H$ is trivial, 
the $H$-support and $H$-weight of an element in
$V\theta_{H}$ is the ordinary support and weight of it, respectively.
The $H$-\emph{inner product} of any two elements 
$u = \sum_{i =  1}^{t} u_{i}\overline{H_{V}(\alpha_{i})}$,
$v = \sum_{i =  1}^{t} v_{i}\overline{H_{V}(\alpha_{i})}$
of $V\theta_{H}$,
can be defined as:
\[
	(u,v)_{H}
	=
	\sum_{i=1}^{t}
	u_{i}v_{i}.
\]
In respect to the $H$-inner product, 
the \emph{left $H$-dual} of a right $R$-submodule~$C$ 
in $V\theta_{H}$ and the \emph{right $H$-dual} of a left 
$R$-submodule~$D$ in $V\theta_{H}$ 
are denoted by $^{\perp_{H}}C$ and $D^{\perp_{H}}$, 
respectively, 
and defined as follows:
\begin{align*}
	^{\perp_{H}}C
	& :=
	\{
	v \in V\theta_{H}
	\mid
	(v,u)_{H} =0
	\text{ for all }
	u \in C
	\}\\
	D^{\perp_{H}}
	& :=
	\{
		v \in V\theta_{H}
		\mid
		(u,v)_{H} =0
		\text{ for all }
		u \in D
	\}.
\end{align*}
If $H$ consists only the identity element, 
then the above mentioned left (resp. right) dual coincide with the ordinary 
left (resp. right) dual of a left (resp. right) $R$-linear code in $V$.

\subsection{Weight enumerators and their MacWilliams identity}

\begin{df}
	Let $R$ be a finite Frobenius ring.
	Assume that $G$ is a finite permutation group on 
	the coordinates of $V (= R^{n})$.
	Let~$H$ be a subgroup of~$G$
	satisfying $|H|$ has an inverse in~$R$. 
	Let $D$ be a left (or right) $R$-submodule in~$V\theta_{H}$.
	Then the \emph{$H$-weight enumerator} of $D$ 
	is defined as:
	\[
		W_{D}^{H}(x,y)
		:=
		\sum_{u\in D}
		x^{t-\wt_{H}(u)}
		y^{\wt_{H}(u)},
	\]
	where $t$ is the number of orbits in the coordinates of $V$ 
	under the action of $H$.
	If $D = C\theta_{H}$ for any left (or right) $G$-code $C$ of 
	length~$n$ over~$R$,
	we prefer to call 
	$W_{C\theta_{H}}^{H}$ as $H$-equivariant weight enumerator of~$C$. 
	Obviously, if $H$ is trivial, then the $H$-equivariant weight enumerator 
	coincide with the ordinary weight enumerator.
	Moreover, if $|G|$ has an inverse in~$R$, then we call $W_{C\theta}^{G}$
	the $G$-equivariant weight enumerator of~$C$.
\end{df}

Honold et al.~\cite{HL2001} first gave the MacWilliams identity 
for the codes over a finite Frobenius ring~$R$. 
Here we can extend the identity for the~$G$-codes over~$R$. 

\begin{thm}[MacWilliams Identity]\label{Thm:EquivMacWilliams}
	Let $R$ be a finite Frobenius ring. 
	Assume that $G$ is a finite permutation group on 
	the coordinates of $R^{n}$.
	Let~$H$ be a subgroup of~$G$
	such that $|H|$ has an inverse in~$R$.
	If~$C$ is a $G$-code of length~$n$ over~$R$,
	then we have the following relation:
	\[
		W_{^{\perp}C\theta_{H}}^{H}
		(x,y)
		=
		\dfrac{1}{|C\theta_{H}|}
		W_{C\theta_{H}}^{H}
		(x+(|R|-1)y,x-y),
	\]
	where $M_{H}$ is the $H$-orbit length matrix.
\end{thm}

If $C$ is a $G$-code of length~$n$ over~$\FF_{q}$
and $H = G$ provided $|G|$ has an inverse in~$\FF_{q}$, 
then the above theorem coincide with the MacWilliams identity given
in Atsumi~\cite[Theorem 1]{Atsumi1995}.
To prove the above MacWilliams identity, 
we need to show the following result.

\begin{thm}\label{Thm:CodeThetaMatrix}
	Let $R$ be a finite Frobenius ring.  
	Assume that $G$ is a finite permutation group on 
	the coordinates of $V (=R^{n})$. 
	Let	$H$ be a subgroup of~$G$ such that
	$|H|$ has an inverse in~$R$.
	If $C$ is a $G$-code of length~$n$ over $R$,
	then we have
	\[
		^{\perp_{H}}(C\theta_{H}) 
		=
		(^{\perp}C\theta_{H})M_{H},
	\]
	where $M_{H}$ is the $H$-orbit length matrix.
\end{thm}

\begin{proof}
	Let 
	$H_{V}(\alpha_{1}),\ldots,H_{V}(\alpha_{t})$ be the orbits of the coordinates of $V$ 
	under the action of $H$. 
	Then the elements of $V\theta_{H}$ can be written as:
	\begin{align*}
		{x} 
		& = 
		\sum_{i = 1}^{t}
		x_{i}
		\overline{H_{V}(\alpha_{i})}
		\in 
		C\theta_{H},\\
		{y} 
		& = 
		\sum_{i = 1}^{t}
		y_{i}
		\overline{H_{V}(\alpha_{i})}
		\in 
		{^\perp} C\theta_{H}.
	\end{align*}
	Let $m_{i}$ be the orbit length of $H_{V}(\alpha_{i})$.
	Then by Theorem~\ref{Thm:ThetaH}, we have
	\begin{equation*}
		0 
		= 
		({y},{x})
		=
		\sum_{i = 1}^{t}
		m_{i}y_{i} x_{i}
		=
		\sum_{i = 1}^{t}
		y_{i}^{\prime}x_{i} 
		=
		({y}^{\prime},{x})_{H},
	\end{equation*}
	where
	$y_{i}^{\prime} = m_{i}y_{i}$
	and
	\[
		{y}^{\prime}
		=
		\sum_{i = 1}^{t}
		y_{i}^{\prime}\overline{H_{V}(\alpha_{i})}.
	\]
	Now we can write
	\[
		{y}^{\prime}
		=
		\sum_{i = 1}^{t}
		m_{i}y_{i}\overline{H_{V}(\alpha_{i})}
		=
		{y}M_{H}.
	\]
	This implies
	\begin{equation}\label{Equ:MGsubset1}
		(^{\perp}C\theta_{H})M_{H}
		\subseteq
		{^{\perp_{H}}(C\theta_{H})}.
	\end{equation}
	Now we need to show that
	$^{\perp_{H}}(C\theta_{H}) \subseteq (^{\perp}C\theta_{H})M_{H}$.
	To do so, let 
	\begin{align*}
		{x} 
		& = 
		\sum_{i = 1}^{t}
		x_{i}
		\overline{H_{V}(\alpha_{i})}
		\in 
		{^{\perp_{H}}(C\theta_{H})},\\
		{y} 
		& = 
		\sum_{i = 1}^{t}
		y_{i}
		\overline{H_{V}(\alpha_{i})}
		\in 
		C\theta_{H}.
	\end{align*}
	For 
	${x}^{\prime} = \sum_{i = 1}^{t} (x_{i}/m_{i})\overline{H_{V}(\alpha_{i})} \in V\theta_{H}$,
	\[
		({x}^{\prime},{y})
		=
		\sum_{i=1}^{t}
		m_{i}
		(x_{i}/m_{i})
		y_{i}
		=
		\sum_{i=1}^{t}
		x_{i}
		y_{i}
		=
		({x},{y})_{H}
		=
		0.
	\]
	Therefore, ${x}^{\prime} \in {^\perp}(C\theta_{H})$.
	This implies together with Theorem~\ref{Thm:ThetaH} that
	${x}^{\prime} \in {^\perp}C\theta_{H}$.
	Hence ${x} = {x}^\prime M_{H} \in (^{\perp}C\theta_{H})M_{H}$.
	This implies that 
	\begin{equation}\label{Equ:MGsubset2}
		^{\perp_{H}}(C\theta_{H}) 
		\subseteq 
		(^{\perp}C\theta_{H})M_{H}.
	\end{equation}
	Therefore by (\ref{Equ:MGsubset1}) and (\ref{Equ:MGsubset2}) we have
	\[
		^{\perp_{H}}(C\theta_{H}) 
		=
		(^{\perp}C\theta_{H})M_{H}.
	\]
	This completes the proof.
\end{proof}

\begin{ex}\label{Ex:Z4Code}
	Let $G = \langle (1,2,3)(4) \rangle$
	be a permutation group with order~$3$ acting on 
	$V = \ZZ_{4}^{4}$ 
	having the orbits $G(\alpha_{1}) = \{1,2,3\}$, $G(\alpha_{2}) = \{4\}$.
	Then
	\[
		\overline{G(\alpha_{1})}
		=
		(1,1,1,0),
		\quad
		\overline{G(\alpha_{2})}
		= 
		(0,0,0,1).
	\] 
	Let $C$ be the $G$-code of length~$4$ over $\ZZ_{4}$
	with the following $16$ codewords:
	\begin{center}
		\begin{tabular}[h!]{cccc}
			0000 & 1113 & 1311 & 2020\\
			0022 & 3331 & 3133 & 0202\\
			2200 & 3111 & 1131 & 2002\\
			0220 & 1333 & 3313 & 2222
		\end{tabular}
	\end{center}
	Since $C$ is self-dual, therefore  $C\theta = {^\perp}C\theta$.
	Then by direct calculation, we have
	\begin{align*}
		C\theta 
		& = 
		\{0000,1113,3331,2222\},\\ 
		^{\perp_{G}}(C\theta) 
		& = 
		\{0000, 1111,2222,3333\}.
	\end{align*}
	It follows that
	$^{\perp_{G}}(C\theta) = ({^\perp}C\theta)M$, 
	where the orbit length matrix $M = \diag(3,3,3,1)$.
	Then the  $G$-equivariant weight enumerator is
	\[
		W_{^{\perp_{G}}(C\theta)}(x,y)
		=
		x^{2} + 3y^{2}
		=
		W_{^{\perp}C\theta}(x,y).
	\]
\end{ex}

\begin{proof}[Proof of Theorem~\ref{Thm:EquivMacWilliams}]
	Note that the MacWilliams identity for the ordinary weight enumerators of codes
	also holds for the codes $C\theta_{H}$  and its $H$-dual in $V\theta_{H}$ as follows:
	\begin{equation*}\label{Equ:MacWilliams1}
		W_{^{\perp_{H}}(C\theta_{H})}^{H}(x,y)
		=
		\dfrac{1}{|C\theta_{H}|}
		W_{C\theta_{H}}^{H}
		(x+(|R|-1)y,x-y).
	\end{equation*}
	Since $W_{(^\perp C\theta_{H})M_{H}}^{H} (x,y)= W_{^\perp C\theta_{H}}^{H}(x,y)$,
	hence by Theorem~\ref{Thm:CodeThetaMatrix} the proof is completed.
\end{proof}


\begin{df}
	Let $R$ be a finite Frobenius ring.
	Assume that $G$ is a finite permutation group on 
	the coordinates of $V (= R^{n})$.
	Let~$H$ be a subgroup of~$G$
	satisfying $|H|$ has an inverse in~$R$. 
	Let $D$ be a left (or right) $R$-submodule in~$V\theta_{H}$. 
	Then the \emph{$H$-complete weight enumerator} of $D$ 
	is defined as:
	\[
		\cwe_{H}
		(D:\{x_{a}\}_{a\in R})
		:=
		\sum_{u\in D}
		\prod_{a \in R}
		x_{a}^{n_{a}^{H}(u)}.
	\]
	Here $n_{a}^{H}(u)$ is the number $i$ 
	such that $a = u_{i}$, where 
	$u_{i}$'s are in form~(\ref{Equ:torbit}).
\end{df}

It is immediate that the $H$-complete weight enumerators 
become the usual \emph{complete weight enumerator}
whenever $H$ is trivial.
If $C$ is a left (or right) $G$-code of length~$n$ over~$R$,
we call $\cwe_{H}(C\theta_{H})$ 
as the \emph{$H$-equivariant complete weight enumerator} of~$C$. 
Moreover, if $|G|$ has an inverse in~$R$, 
then we call $\cwe_{G}(C\theta)$
the \emph{$G$-equivariant complete weight enumerator} of~$C$.

\begin{rem}
	We have the following straightforward relations:
	\begin{itemize}
		\item [(1)]
		$\cwe_{H}(D:x_{0} \leftarrow x, x_{a} \leftarrow y \text{ for } a \in R \text{ and } a \neq 0) = W_{D}^{H}(x,y)$.
		
		\item [(2)]
		$\cwe_{H}(C\theta_{H}:x_{0} \leftarrow x, x_{a} \leftarrow y \text{ for } a \in R \text{ and } a \neq 0) = W_{C\theta_{H}}^{H}(x,y)$.
		
		\item [(3)]
		$\cwe_{G}(C\theta:x_{0} \leftarrow x, x_{a} \leftarrow y \text{ for } a \in R \text{ and } a \neq 0) = W_{C\theta}^{G}(x,y)$.
	\end{itemize}
	
\end{rem}



Though 
$W_{(^\perp C\theta_{H})M_{H}}^{H} (x,y)= W_{^\perp C\theta_{H}}^{H}(x,y)$
for $H$-equivariant weight enumerators of $G$-codes over finite Frobenius rings,
but we can observe from Example~\ref{Ex:Z4Code} that 
$(^\perp C\theta_{H})M_{H} \neq {^\perp}C\theta_{H}$. 
This observation concludes a generalization of the MacWilliams identity 
for the equivariant complete weight enumerator for $G$-codes over 
finite Frobenius rings as follows.

\begin{thm}\label{Thm:CompWeightMacWilliams}
	Let $R$ be a finite Frobenius ring. 
	Assume that $G$ is a finite permutation group 
	acts on the coordinates of $R^{n}$.
	Let~$H$ be a subgroup of~$G$
	such that $|H|$ has an inverse in~$R$.
	If $C$ is a $G$-code of length~$n$ over~$R$,
	then we have the relation:
	\[
		\cwe_{H}
		(({^{\perp}C\theta_{H}})M_{H}: \{x_{a}\}_{a \in R})\\
		=
		\dfrac{1}{|C\theta_{H}|} 
		\cwe_{H}
		\left( 
		{C\theta_{H}}:
		\left\{
		\sum_{b \in R}
		\chi\left(a b\right) 
		x_{b}
		\right\}_{a \in R}
		\right),
	\]
	where $M_{H}$ is the $H$-orbit length matrix and $\chi$ is a generating character associated with~$R$.
\end{thm}

\begin{proof}
	Since the MacWilliams identity for usual complete weight enumerators
	of codes over finite Frobenius rings also hold for $C\theta_{H}$ 
	and its $H$-dual
	in $V\theta_{H}$, where $V = R^{n}$, therefore by 
	Theorem~\ref{Thm:CodeThetaMatrix}, we can have the proof.
\end{proof}


\begin{df}
	Let $R$ be a finite Frobenius ring.
	Assume that $G$ is a finite permutation group on 
	the coordinates of $V (= R^{n})$.
	Let~$H$ be a subgroup of~$G$
	satisfying $|H|$ has an inverse in~$R$. 
	Let $D$ be a left (or right) $R$-submodule in~$V\theta_{H}$.
	Then the \emph{ $g$-th $H$-complete weight enumerator} of $D$ 
	is defined as:
	\[
	\cwe_{H}^{(g)}
	(D:\{x_{a}\}_{a\in R^{g}})
	:=
	\sum_{u_{1},\ldots,u_{g}\in D}
	\prod_{a \in R^{g}}
	x_{a}^{n_{a}^{H}(u_{1},\ldots,u_{g})}.
	\]
	Here $n_{a}^{H}(u_{1},\ldots,u_{g})$ is the number $i$ 
	such that $a = (u_{1,i},\ldots,u_{g,i})$, where 
	$u_{i}$'s are in form~(\ref{Equ:torbit}).
	Obviously, if $H$ is trivial, 
	then the  $g$-th $H$-complete weight enumerator 
	becomes the \emph{$g$-th complete weight enumerator}.
	For detail discussions on $g$-th complete weight enumerator,
	we refer the readers to~\cite{CM2021, CMO2022, MO2019}.
	If $C$ is a left (or right) $G$-code of length~$n$ over~$R$,
	we call $\cwe_{H}^{(g)}(C\theta_{H})$ 
	as the \emph{ $g$-th $H$-equivariant complete weight enumerator} of~$C$. 
	Moreover, if $|G|$ has an inverse in~$R$, 
	then we call $\cwe_{G}^{(g)}(C\theta)$
	the \emph{$g$-th $G$-equivariant complete weight enumerator} of~$C$.
\end{df}

\begin{rem}
	We have the following obvious correspondence:
	\begin{itemize}
		\item [(1)]
		$\cwe_{H}^{(1)}(D) = \cwe_{H}(D)$.
		
		\item [(2)]
		$\cwe_{H}^{(1)}(C\theta_{H}) = \cwe_{H}(C\theta_{H})$.
		
		\item [(3)]
		$\cwe_{G}^{(1)}(C\theta) = \cwe_{G}(C\theta)$.
	\end{itemize}
	
\end{rem}



Now we have the genus~$g$ generalization of the MacWilliams identity 
for the equivariant complete weight enumerator for $G$-codes over 
finite Frobenius rings.

\begin{thm}\label{Thm:gthCWEMacWilliams}
	Let $R$ be a finite Frobenius ring. 
	Assume that $G$ is a finite permutation group 
	acts on the coordinates of $R^{n}$.
	Let~$H$ be a subgroup of~$G$ acting on~$R^{n}$
	such that $|H|$ has an inverse in~$R$.
	If $C$ is a $G$-code of length~$n$ over~$R$,
	then we have the relation:
	\[
		\cwe_{H}^{(g)}
		(({^{\perp}C\theta_{H}})M_{H}: \{x_{a}\}_{a \in R^{g}})\\
		=
		\dfrac{1}{|C\theta_{H}|^{g}} 
		\cwe_{H}^{(g)}
		\left( 
		{C\theta_{H}}:
		\left\{
		\sum_{b \in R^g}
		\chi\left(\sum_{i = 1}^{g}a_i b_i\right) 
		x_{b}
		\right\}_{a \in R^{g}}
		\right),
	\]
	where $M_{H}$ is the $H$-orbit length matrix and $\chi$ is a generating character associated with~$R$.
\end{thm}

\begin{proof}
	Since the MacWilliams identity for usual genus~$g$ complete weight enumerators
	of codes over finite Frobenius rings also hold for $C\theta_{H}$ 
	and its $H$-dual
	in $V\theta_{H}$, where $V = R^{n}$, therefore by 
	Theorem~\ref{Thm:CodeThetaMatrix}, we can have the proof.
\end{proof}

\subsection{$G$-code analogue of Bachoc's MacWilliams identity}

Bachoc~\cite{Bachoc} presented the MacWilliams-type identity for
the weight enumerators of binary codes associated with
a discrete harmonic function. This weight enumerator is
known as harmonic weight enumerator. 
Later, Bachoc~\cite{BachocNonBinary}
and Tanabe~\cite{Tanabe2001} independently proved 
the non-binary version of the MacWilliams identity for
the harmonic weight enumerators of codes over finite fields.
Moreover, Britz et al.~\cite{BrChIsMiTa2023} generalized 
this identity for codes over finite Frobenius rings.
Our aim is to present a~$G$-code version of Bachoc's MacWilliams identity 
for codes over finite Frobenius rings.

Let $[t] := \{1,2,\ldots,t\}$, where $t$ is a positive integer.
We denote the set of all $d$-subset of $[t]$ by $[t]_{d}$ 
for $d = 0,1, \ldots, t$. 
Let $R$ be a finite Frobenius ring.  
Assume that $G$ is a finite permutation group on 
the coordinates of $V (=R^{n})$. 
Let	$H$ be a subgroup of~$G$ acting on~$V$ 
with~$t$ orbits such that $|H|$ has an inverse in~$R$.
We denote by 
$\RR (V\theta_{H})$ and $\RR [t]_{d}$
the real vector spaces spanned by the elements of  
$V\theta_{H}$ and $[t]_{d}$,
respectively. 
An element of 
$\RR [t]_{d}$
is denoted by
\begin{equation*}\label{Equ:FunREd}
	f :=
	\sum_{z \in [t]_{d}}
	f(z) z
\end{equation*}
with coefficients $f(z) \in \RR$. 
Thus $\RR [t]_{d}$ is identified with the 
real-valued function on $[t]_{d}$ given by 
$z \mapsto f(z)$. 
Let $(V\theta_{H})_{d} := \{u \in V\theta_{H} \mid \wt_{H}(u) = d\}$. 
An element 
$f \in \RR [t]_{d}$
can be extended to an element 
$\widetilde{f}\in \RR (V\theta_{H})$
by setting, for all 
$u \in V\theta_{H}$,
\begin{equation*}\label{Equ:AcuteF}
	\widetilde{f}(u)
	:=
	\sum_{\substack{v\in (V\theta_{H})_{d},\\ \supp_{H}(v)\subseteq \supp_{H}(u)}}
	f(\supp_{H}(v)).
\end{equation*}
Note that 
$\widetilde{f}(u) = 0$
for any $u \in V\theta_{H}$ such that $\wt_{H}(u) < d$. 
The differentiation $\gamma$ is the operator on $\RR [t]_{d}$ defined by 
linearity from the identity 
\begin{equation*}\label{Equ:Gamma}
	\gamma(z) := 
	\sum_{y\in {[t]}_{d-1}, y\subseteq z} 
	y
\end{equation*}
for all 
$z \in [t]_{d}$
and for all $d=0,1, \ldots t$. 
Also, $\Harm_{d}(t)$ is the kernel of~$\gamma$:

\[
\Harm_{d}(t)
:= 
\ker
\left(
\gamma\big|_{\RR [t]_{d}}
\right).
\]

\begin{df}
	Let $R$ be a finite Frobenius ring.
	Assume that $G$ is a finite permutation group on 
	the coordinates of $V (= R^{n})$.
	Let~$H$ be a subgroup of~$G$ with $t$ orbits
	such that $|H|$ has an inverse in~$R$. 
	Let $D$ be a left (or right) $R$-submodule in~$V\theta_{H}$.
	Let $f \in \Harm_{d}(t)$.
	Then the \emph{harmonic $H$-weight enumerator} of $D$ 
	associated to~$f$ is defined as:
	\[
		W_{D,f}^{H}(x,y)
		:=
		\sum_{u\in D}
		\widetilde{f}(u)
		x^{t-\wt_{H}(u)}
		y^{\wt_{H}(u)}.
	\]
\end{df}
Clearly, if $H$ is trivial, then the harmonic $H$-weight enumerators 
become the harmonic weight enumerators.
We refer the readers to~\cite{Bachoc, BrChIsMiTa2023, CMO2023}
for the details of harmonic weight enumerators.
If $D = C\theta_{H}$ for any left (or right) $G$-code $C$ of 
length~$n$ over~$R$,
we call 
$W_{C\theta_{H},f}^{H}$ the harmonic $H$-equivariant weight enumerator of~$C$. 
Moreover, if $|G|$ has an inverse in~$R$, then we call $W_{C\theta,f}^{G}$
the harmonic $G$-equivariant weight enumerator of~$C$.


 
Now we give the MacWilliams identity for harmonic 
weight enumerators of $G$-codes over finite Frobenius rings.

\begin{thm}\label{Thm:HarmMacWilliams} 
	Let $R$ be a finite Frobenius ring. 
	Assume that $G$ is a finite permutation group on 
	the coordinates of $R^{n}$.
	Let $H$ be a subgroup of~$G$
	with~$t$ orbits	such that $|H|$ 
	has an inverse in~$R$.
	Let $f \in \Harm_{d}(t)$.
	Let $C$ be a left $R$-linear code of length~$n$.
	Then we have
	\[
		W_{C\theta_{H},f}^{H}(x,y) 
		:= 
		(xy)^{d} Z_{C\theta_{H},f}^{H}(x,y),
	\]
	where $Z_{C\theta_{H},f}^{H}$ is a homogeneous polynomial of degree $t-2d$ 
	and satisfies
	\[
		Z_{(^{\perp}C\theta_{H})M_{H},f}^{H}
		(x,y)
		= 
		(-1)^{d} 
		\frac{|R|^{d}}{|C\theta_{H}|} 
		Z_{C\theta_{H},f}^{H} 
		\left( 
		x+(|R|-1)y, 
		x-y 
		\right),
	\]
	where $M_{H}$ is the $H$-orbit length matrix.
\end{thm}

\begin{proof}
	Harmonic generalization of the MacWilliams identity for code over 
	finite Frobenius rings also hold for $C\theta_{H}$ and its $H$-dual in
	$V\theta_{H}$, where $V = R^{n}$. So, by Theorem~\ref{Thm:CodeThetaMatrix},
	we can complete the proof.
\end{proof}

\section{Lattices with group actions}\label{Sec:Glattices}

Let $\Gamma$ be the real $n$-dimensional vector space~$\RR^{n}$. 
Let $G$ be a finite permutation group 
acting on the coordinates of~$\Gamma$. 
Then $\Gamma$ is an $\RR G$-\emph{module} 
if there exists a mapping  
\[
	G \times \Gamma \to \Gamma;
	(g,v) \mapsto vg
\]
satisfying 
$vg \in \Gamma$ 
for all~$g \in G$ and $v \in \Gamma$.
Let~$v = (v_{1},\ldots,v_{n}) \in \Gamma$ and $g \in G$.
We define the action of $G$ on~$\Gamma$ as follows:
\[
	vg
	:=
	(x_{1},\ldots,x_{n})
	\mbox{ such that }
	x_{i} = v_{ig^{-1}}
	\mbox{ for }
	i = 1,\ldots,n.
\]
In this way, 
$\Gamma$ becomes an $\RR G$-module. 
A $G$-\emph{lattice} is a lattice which is also $\ZZ G$-submodule of $\Gamma$.
For any subgroup~$H$ of~$G$,
we define an operator:
\[
	\theta_{H}
	:=
	\frac{1}{|H|}
	\sum_{h\in H}
	h.
\]
Then $\Gamma\theta_{H} := \{v\theta_{H} \mid v \in \Gamma\}$.
It is immediate that 
$\Gamma\theta_{H} = \{v \in \Gamma \mid vg = v \mbox{ for all } g \in G\}$.
Clearly, $\Gamma\theta_{H}$ is an $\RR$-submodule of~$\Gamma$. 
Let
$H_{\Gamma}(\alpha_{1}),\ldots,H_{\Gamma}(\alpha_{t})$ 
be the orbits of the coordinates of $\Gamma$ under the action of $H$.
Let $m_{i}$ be the length of $H_{\Gamma}(\alpha_{i})$.
Define $\overline{H_{\Gamma}(\alpha_{i})}$ as the vector of~$\Gamma$
which has~$1/\sqrt{m_{i}}$ as its entry for every point of~$H_{\Gamma}(\alpha_{i})$
and~$0$ elsewhere.

Therefore, each element~$u \in \Gamma\theta_{H}$ 
can be written in the following form:
\begin{equation}\label{Equ:Lattorbit}
	u	
	= 
	\sum_{i = 1}^{t}	
	u_{i} \overline{H_{\Gamma}(\alpha_{i})}.
\end{equation}
This implies that $\Gamma\theta_{H}$ is a~$t$-dimensional vector space. 
The $H$-\emph{inner product} of any two elements 
$u = \sum_{i =  1}^{t} u_{i}\overline{H_{\Gamma}(\alpha_{i})}$,
$v = \sum_{i =  1}^{t} v_{i}\overline{H_{\Gamma}(\alpha_{i})}$
of $\Gamma\theta_{H}$,
can be defined as:
\[
\la u,v\ra_{H}
=
\sum_{i=1}^{t}
u_{i}v_{i}.
\]
Let~$D$ be a lattice in $\Gamma\theta_{H}$.
That is, there exists an~$\RR$-basis consisting of~$t$ 
elements of~$\Gamma\theta_{H}$ which is a~$\ZZ$-basis of~$D$.
The \emph{dual lattice} of~$D$ with respect to the $H$-inner product, 
is denoted by $D^{\ast_{H}}$, 
and defined as follows:
\begin{align*}
	D^{\ast_{H}}
	& :=
	\{
	v \in \Gamma\theta_{H}
	\mid
	\la u,v\ra_{H} \in \ZZ
	\text{ for all }
	u \in D
	\}.
\end{align*}
If $H$ consists only the identity element, 
then the above mentioned dual coincide with the ordinary 
dual of a lattice in $\Gamma$.

Now we have the following useful lemma:

\begin{lem}\label{Lem:LambdaZeroTheta}
	Let $\Lambda$ be a $G$-lattice, 
	where $G$ is a finite permutation group on 
	the coordinates of $\Gamma$.
	For any subgroup~$H$ of~$G$,
	we define:
	$\Lambda_{0}:=\{v \in \Lambda \mid v\theta_{H} \in \Lambda\}$.
	Then $\Lambda_{0}\theta_{H}$ is a lattice.
\end{lem}
Now we have the lattice analogue of 
Theorem~\ref{Thm:ThetaH}
as follows.

\begin{thm}\label{Thm:LatticeHayden}
	Let $\Lambda$ be a $G$-lattice, 
	where $G$ is a finite permutation group on 
	the coordinates of $\Gamma$.
	For any subgroup~$H$ of~$G$
	 we have
	\[
		(\Lambda_{0}\theta_{H})^{\ast}
		=
		\ker\theta_{H} 
		\oplus 
		\Lambda_{0}^\ast
		\theta_{H}.
	\]
\end{thm}

\begin{proof}
	The proof is similar to the proof of~\cite[Theorem 4.2]{BrHalHay1981},
	so we omit it.
\end{proof}

\begin{df}
	Let $D$ be a lattice in~$\Gamma\theta_{H}$, 
	where~$H$ is a subgroup of a permutation group~$G$ acting
	on the coordinates of~$\Gamma$. 
	Then the \emph{H-theta series} of~$D$ in~$\Gamma\theta_{H}$ 
	is defined as follows:
	\[
		\Theta_{D}^{H}(z)
		:=
		\sum_{x \in D}
		q^{\la x, x \ra_{H}},
	\]
	where $q := e^{\pi i z}$ and $z \in \HH$,
	and	$x$ is in form~(\ref{Equ:Lattorbit}).
	Clearly, when $H$ is trivial, 
	the $H$-theta series 
	becomes the usual theta series.
\end{df}

Now we have the following $G$-lattice version of Jacobi's formula for theta series.

\begin{thm}\label{Thm:GLatticeJacobiFormula}
	Let $\Lambda$ be a $G$-lattice, 
	where $G$ is a finite permutation group on 
	the coordinates of $\Gamma$.
	Then for any subgroup~$H$ of~$G$
	the following relation:
	\[
		\vartheta_{(\Lambda_{0}\theta_{H})^{\ast_{H}}}(z)
		=
		(\det \Lambda_{0}\theta_{H})^{1/2}
		(i/z)^{t/2}
		\vartheta_{\Lambda_{0}\theta_{H}}(- 1/z).
	\]
\end{thm}

\section{From $G$-codes to $G$-lattices}\label{Sec:GcodetoGlattice}

There is a lot of research that established the relationship 
between codes and lattices, for instance see~\cite{BDHO1999,Munemasa,Runge}.
The well known connection between the weight enumerators
of codes and the theta series of lattices has been formed
in these literatures using special types of theta functions.
In this section,
we extend this connection for higher genus cases by presenting 
an analogous relationship between the 
complete weight enumerators of $G$-codes and 
the theta series of $G$-lattices.
To do so, 
we restrict our discussions over the class of finite Frobenius rings
that includes only the ring~$\ZZ_{k}$ of integers modulo~$k \geq 2$ 
and finite field~$\FF_{p}$ of order~$p$, where~$p$ is prime.
More precisely,
by finite Frobenius ring~$R_{0}$,
we mean either~$\ZZ_{k}$ or~$\FF_{p}$.
For notation and terminology used in this section,
we refer the readers to the previous sections.

Let $R_{0}$ be a finite Frobenius ring.
Let~$C$ be a linear code of length~$n$ over~$R_{0}$. 
Then we can construct a lattice~$\Lambda(C)$ from~$C$ 
by the well-known Construction~$A$ as follows 
\[
	\Lambda(C)
	:=
	\frac{1}{\sqrt{|R_{0}|}}
	\{
		x \in \ZZ^{n}
		\mid
		\rho(x) \in C
	\},
\]
 where the map
$\rho : \ZZ^{n} \to R_{0}^{n}$ defined by
$$\rho(x_{1},\ldots,x_{n}) = (x_{1}\pmod{|R_{0}|},\ldots,x_{n}\pmod{|R_{0}|}).$$

\begin{thm}\label{Thm:GcodeGlattice1}
	Let $C$ be a linear code of length~$n$ over finite Frobenius ring~$R_{0}$. 
	Assume that $G$ is a finite permutation group on 
	the coordinates of $R_{0}^{n}$. 
	Then~$C$ is a $G$-code if and only if $\Lambda(C)$ is a $G$-lattice.
\end{thm}

\begin{proof}
	Let $\Lambda(C)$ be a $G$-lattice.
	Let $x \in \Lambda(C)$. 
	Then $x$ can be written in the form:
	$$x = \dfrac{1}{\sqrt{|R_{0}|}} (c+|R_{0}|z) \in \Lambda(C),$$
	for $c \in C$ and $z \in \ZZ^{n}$.
	Then for any $g \in G$, 
	$$xg = \frac{1}{\sqrt{|R_{0}|}}(cg+|R_{0}|zg) = \frac{1}{\sqrt{|R_{0}|}}(c'+|R_{0}|z'),$$ 
	where $c' = cg$ and $z' = zg $.
	Since $z' \in \ZZ^{n}$, therefore $ xg\in \Lambda(C)$
	if only if $c' \in C$. 
	Hence~$C$ is a $G$-code if and only if $\Lambda(C)$ is a $G$-lattice.
\end{proof}

\begin{thm}\label{Thm:GcodeGlattice2}
	Let $R_{0}$ be a finite Frobenius ring. 
	Let $C$ be a $G$-code, 
	where $G$ is a finite permutation group on 
	the coordinates of $R_{0}^{n}$.
	Let	$H$ be a subgroup of~$G$ 
	such that~$|H|$ has an inverse in~$R_{0}$.
	Define
	$\Lambda_{0}(C):=\{v \in \Lambda(C) \mid v\theta_{H} \in \Lambda(C)\}$.
	Then $\Lambda_{0}(C)\theta_{H}$ is a lattice.
	Moreover, $\Lambda_{0}(C)\theta_{H} = \Lambda(C\theta_{H})$.
\end{thm}

\begin{proof}
	By using Lemma~\ref{Lem:LambdaZeroTheta}
	and Theorem~\ref{Thm:GcodeGlattice1}, 
	it is straightforward to show that 
	$\Lambda_{0}(C)\theta_{H}$ is a lattice.
	And the remaining part is easy to proof.
\end{proof}

\subsection{Theta series of lattices with group actions}

We recall~\cite{Runge} for the notations of theta functions. 
Let~$g$ be a positive integer.
Define 
$\HH_{g}:=\{\tau \in \Mat_{g \times g}(\CC) \mid \tau \text{ is symmetric, }\Ima(\tau)>0\}$.
For $a \in R_{0}^{g}$
and~$\tau \in \HH_{g}$,
we take the theta functions as follows:

\begin{align*}
	f_{a}(\tau)
	:=
	\sum_{\substack{b \in \ZZ^{g}\\ b \equiv a \pmod{|R_{0}|}}}
	\exp 
	\left(
	\frac{1}{|R_{0}|}
	\pi {i}
	(^{t}b\tau b)
	\right).	
\end{align*}

Moreover, the \emph{$g$-th theta series} of a lattice~$\Lambda$
is defined as follows:
\[
	\vartheta_{\Lambda}^{(g)}(\tau)
	:=
	\sum_{x_{1},\ldots,x_{g} \in \Lambda}
	\exp\left( \pi {i}
	\tr(\tau(\la x_{j},x_{k}\ra)_{1\leq j,k\leq g})\right), 
\]
where $\tau \in \HH_{g}$.

\begin{df}
	Let $G$ be a finite permutation group on 
	the coordinates of $\Gamma$. 
	Let $D$ be a lattice of~$\Gamma\theta_{H}$,
	where~$H$ is a subgroup of~$G$.
	Then the \emph{$g$-th $H$-theta series} of $D$ 
	is defined as:
	\[
		\vartheta_{D}^{H,(g)}
		(\tau)
		:=
		\sum_{x_{1},\ldots,x_{g} \in D}
		\exp \left(\pi {i}
		\tr(\tau(\la x_{j},x_{k}\ra_{H})_{1\leq j,k\leq g})\right).
	\]
	where 
	$\tau \in \HH_{g}$ and $x_{i}$'s are in form~(\ref{Equ:Lattorbit}).
	Obviously, if $H$ is trivial, 
	then the  $g$-th $H$-theta series 
	becomes the \emph{$g$-th theta series}.
	Moreover, when $g=1$ considering $H$ is trivial, 
	can have the ordinary theta series. 
\end{df}

\begin{thm}\label{Thm:WeightEnumtoThetaSeries}
	Let $C$ be a $G$-code over~$R_{0}$, 
	where $G$ is a finite permutation group on 
	the coordinates of~$V(=R_{0}^{n})$.
	Let	$H$ be a subgroup of~$G$ 
	such that~$|H|$ has an inverse in~$R_{0}$. 
	Then
	\[
		\vartheta_{\Lambda_{0}(C)\theta_{H}}^{H,(g)}(\tau)
		=
		\cwe_{H}^{(g)}
		(C\theta_{H}:{f_{a}(\tau)}_{a \in R_{0}^{g}}).
	\]
\end{thm}

\begin{proof}
	Let $C$ be a $G$-code over~$R_{0}$.
	Assume that $H$ be any subset of~$G$ such that~$|H|$
	has an inverse in~$R_{0}$. 
	Let
	$H_{V}(\alpha_{1}),\ldots,H_{V}(\alpha_{t})$ 
	be the orbits of the coordinates of $V$ under the action of $H$.
	Then each element~${c} \in C\theta_{H}$ 
	can be written in the following form:
	\begin{equation*}
		c	
		= 
		\sum_{j = 1}^{t}	
		c_{j} \overline{H_{V}(\alpha_{j})}.
	\end{equation*}
	Let $a_{1}, \ldots, a_{t} \in R_{0}^{g}$ such that
	\[
		\begin{pmatrix}
			a_{1} & \cdots & a_{t}
		\end{pmatrix}
		=
		\begin{pmatrix}
			u'_{1}\\
			\vdots\\
			u'_{g}
		\end{pmatrix}
		\in
		\Mat_{g\times t}
		(R_{0}),
	\]
	where the row vector $u'_{k} = (u'_{k1},\ldots,u'_{kt})$
	satisfies the condition that
	\begin{equation}\label{Equ:CodetoLattice}
		{u}_{k}	
		= 
		\sum_{j = 1}^{t}	
		u'_{kj} \overline{H_{V}(\alpha_{j})} \in C\theta_{H}.
	\end{equation}
Now
\begin{align*}
	\prod_{j=1}^{t}
	f_{a_{j}}(\tau)
	& =
	\prod_{j=1}^{t}
	\sum_{\substack{b \in \ZZ^{g}\\ b \equiv a_{j} \pmod{|R_{0}|}}}
	\exp
	\left(
	\frac{1}{|R_{0}|} 
	\pi {i}
	(^{t}b\tau b)
	\right)\\
	& =
	\sum_{\substack{b_{j} \in \ZZ^{g}\\ b_{j} \equiv a_{j} \pmod{|R_{0}|}\\ (1 \leq j \leq t)}}
	\prod_{j=1}^{t}
	\exp 
	\left(
	\frac{1}{|R_{0}|}
	\pi {i}
	(^{t}b_{j}\tau b_{j})
	\right)\\
	& =
	\sum_{\substack{b_{j} \in \ZZ^{g}\\ b_{j} \equiv a_{j} \pmod{|R_{0}|}\\ (1 \leq j \leq t)}}
	\exp \pi {i}
	\sum_{j=1}^{t}
	\left(\frac{1}{|R_{0}|}\tr(^{t}b_{j}\tau b_{j})\right)\\
	& =
	\sum_{\substack{b_{j} \in \ZZ^{g}\\ b_{j} \equiv a_{j} \pmod{|R_{0}|}\\ (1 \leq j \leq t)}}
	\exp \pi {i}
	\left(\frac{1}{|R_{0}|}\tr(\tau \sum_{j=1}^{t} {^{t}b_{j}} b_{j})\right)\\
	& =	
	\sum_{\substack{\alpha_{j} \in \ZZ^{n}\\ \alpha_{j} \equiv u_{j} \pmod{|R_{0}|}\\ (1 \leq j \leq g)}}
	\exp \pi {i}
	\left(\frac{1}{|R_{0}|}\tr(\tau (\la\alpha_{j},\alpha_{k}\ra_{H}))\right)\\
	& =
	\sum_{\substack{x_{j} \in \rho^{-1}(u_{j})\\ (1 \leq j \leq g)}}
	\exp \pi {i}
	\left(\tr(\tau (\la x_{j}, x_{k}\ra_{H}))\right).
\end{align*}
	Summing all such $(u_{1},\ldots,u_{g}) \in C\theta_{H}$, 
	and by Theorem~\ref{Thm:GcodeGlattice2}, we have
	\begin{align*}
		\cwe_{H}^{(g)}(C\theta_{H}:\{f_{a}(\tau)\}_{a\in R_{0}^{g}})
		& =
		\sum_{{u}_{1},\ldots,{u}_{g} \in C}
		\sum_{\substack{x_{j} \in \rho^{-1}(u_{j})\\ (1 \leq j \leq g)}}
		\exp\left(\tr(\tau (\la x_{j}, x_{k}\ra_{H}))_{1\leq j,k \leq g}\right)\\
		& =
		\vartheta_{\Lambda_{0}(C)\theta_{H}}^{H,(g)}(\tau)
	\end{align*}
	This completes the proof.
\end{proof}

\section{Jacobi polynomials and Jacobi theta series}\label{Sec:JacobiPolyTheta}

In this section, 
we present the Jacobi generalization of Astumi's MacWilliams identity.
Moreover, we give an analogue Theorem~\ref{Thm:WeightEnumtoThetaSeries}
for the complete Jacobi
polynomials of codes over finite Frobenius rings 
with group action.

\begin{df}\label{Def:CompleteJac}
	Let $R$ be a finite Frobenius ring.
	Assume that $G$ is a finite permutation group on 
	the coordinates of $V (= R^{n})$.
	Let~$H$ be a subgroup of~$G$
	satisfying $|H|$ has an inverse in~$R$. 
	Let $D$ be a left (or right) $R$-submodule in~$V\theta_{H}$.
	Then the \emph{complete $H$-Jacobi polynomial} 
	attached to a set $T \subseteq E$ of 
	coordinate places of~$D$ is defined as:
	\[
	\CJ_{D,T}^{H}(\{x_{a},y_{a}\}_{a \in R}) 
	:=
	\sum_{u\in D}
	\prod_{a \in R}
	x_{a}^{n_{a,T}^{H}(u)}
	y_{a}^{n_{a, E\backslash T}^{H}(u)},
	\]
	where $n_{a,X}^{H}(u)$ is the number of~$i$'s 
	in~$X \subseteq E$ such that $a= u_{i}$ 
	considering $u$ as in form~(\ref{Equ:torbit}).
	Obviously, if $H$ is trivial, 
	then the complete $H$-Jacobi polynomials 
	become the \emph{complete Jacobi polynomials}.
	For detail discussions on the Jacobi polynomials,
	we refer the readers to~\cite{CM2022,CMO2022,CMOT2023,CIT20xx,Ozeki}.
	If $C$ is a left (or right) $G$-code of length~$n$ over~$R$,
	we call $\CJ_{C\theta_{H},T}^{H}$ 
	as the \emph{ $H$-equivariant complete Jacobi polynomial} of~$C$. 
	Moreover, if $|G|$ has an inverse in~$R$, 
	then we call $\CJ_{C\theta,T}^{G}$
	the \emph{ $G$-equivariant complete Jacobi polynomial} of~$C$.
\end{df}

The MacWilliams-type identity for the Jacobi polynomials 
of codes over the finite fields 
was first given by Ozeki~\cite{Ozeki}. 
Now we have a generalization of the MacWilliams identity 
for the equivariant complete Jacobi polynomials of $G$-codes over a 
finite Frobenius ring.

\begin{thm}\label{Thm:EquivJacobiMacWilliams}
	Let $R$ be a finite Frobenius ring. 
	Assume that $G$ is a finite permutation group 
	acts on the coordinates of $R^{n}$.
	Let~$H$ be a subgroup of~$G$ 
	such that $|H|$ has an inverse in~$R$.
	If $C$ is a $G$-code of length~$n$ over~$R$,
	then we have the relation:
	\begin{align*}
		\CJ_{({^{\perp}C\theta_{H}})M_{H},T}^{H}
		&(\{x_{a},y_{a}\}_{a \in R})\\
		& =
		\dfrac{1}{|C\theta_{H}|} 
		\CJ_{{C\theta_{H}},T}^{H}
		\left( 
		\left\{
		\sum_{b \in R}
		\chi
		\left(
		a b
		\right) 
		x_{b}
		\right\}_{a \in R},
		\left\{
		\sum_{b \in R}
		\chi
		\left(
		a b
		\right) 
		y_{b}
		\right\}_{a \in R}
		\right),
	\end{align*}
	where $M_{H}$ is the $H$-orbit length matrix and 
	$\chi$ is a generating character associated with~$R$.
\end{thm}

\begin{proof}
	The proof of the MacWilliams identity for the complete Jacobi polynomials
	of  $C\theta_{H}$ and its $H$-dual
	in $V\theta_{H}$ over finite Frobenius rings, where $V = R^{n}$, 
	is straightforward. Therefore by 
	Theorem~\ref{Thm:CodeThetaMatrix}, we can have the result.
\end{proof}

\begin{df}
	Let $G$ be a finite permutation group on 
	the coordinates of $\Gamma$. 
	Let $D$ be a lattice of~$\Gamma\theta_{H}$,
	where~$H$ is a subgroup of~$G$.
	Then the \emph{$H$-Jacobi theta series} of $D$ 
	with respect to $y \in D$ can be defined as:
	\[
		\vartheta_{\Lambda,y}^{H}(\tau, z)
		:=
		\sum_{x \in \Lambda}
		\exp
		\left( 
		\pi {i}
		\tau\la x,x\ra_{H}
		+
		2 {\pi} i
		z\la x,y\ra_{H}
		\right),
	\]
	where 
	$\tau \in \HH$, $z \in \CC$,
	and $x$ and $y$ are in form~(\ref{Equ:Lattorbit}).
\end{df}

Clearly, if $H$ is trivial, 
then the $H$-Jacobi theta series 
becomes the {ordinary Jacobi theta series}
introduced by Bannai and Ozeki~\cite{BO1996}. 
Moreover, they gave a correspondence between
the Jacobi polynomials of codes and Jacobi theta series
of lattices. In the following result, we give a $G$-code
analogue of the Bannai-Ozeki correspondence.
The proof of the theorem is straightforward.
So, we leave the proof for the readers.

\begin{thm}\label{Thm:JacPolytoThetaSeries}
Let $R_{0}$ be a finite Frobenius ring.
For $a \in R_{0}$
and~$\tau \in \HH$ and $z \in \CC$,
we take the following theta functions:

\begin{align*}
	\psi_{a}(\tau)
	&:=
	\sum_{\substack{b \in \ZZ\\ b \equiv a \pmod{|R_{0}|}}}
	\exp 
	\left(
	\frac{1}{|R_{0}|}
	\pi {i}\tau
	\la b, b\ra
	\right),\\
	\phi_{a}(\tau,z)
	& :=
	\sum_{\substack{b \in \ZZ\\ b \equiv a \pmod{|R_{0}|}}}
	\exp 
	\left(
	\frac{1}{|R_{0}|}
	\pi {i} \tau
	\la b,b \ra
	+
	2{\pi} i zb		
	\right).	
\end{align*}
Let $C$ be a $G$-code over~$R_{0}$, 
where $G$ is a finite permutation group on 
the coordinates of~$R_{0}^{n}$.
Let	$H$ be a subgroup of~$G$ 
such that~$|H|$ has an inverse in~$R_{0}$. Then
\[
	\CJ_{C\theta_{H},T}^{H}(\{x_{a} \leftarrow \psi_{a}(\tau),y_{a} \leftarrow \phi_{a}(\tau,z)\}_{a \in R_{0}})
\]
is an $H$-Jacobi theta series.
\end{thm}

\section{Concluding remarks}

Cameron~\cite{Cameron} first showed a correspondence between the cycle indices and 
the weight enumerators. Subsequently, the correspondence was investigated for higher genus
cases in~\cite{CMO2024, MO2019}. So, it is natural to find an analogue of the
relation for $G$-codes with an action of Hayden's operator.
We shall investigate this analogous relation in some subsequent papers.


\section*{Acknowledgements}

The authors thank Manabu Oura for helpful discussions. 
The first named author is supported by SUST Research Centre under 
Research Grant PS/2023/1/22.
The second named author is supported by JSPS KAKENHI (22K03277). 


\end{document}